\documentclass[10pt]{article}
\font\smallit=cmti10

\usepackage{amssymb,latexsym,amsmath,epsfig,amsthm} %% Add other packages as necessary

\makeatletter

\renewcommand\section{\@startsection {section}{1}{\z@}
{-30pt \@plus -1ex \@minus -.2ex}
{2.3ex \@plus.2ex}
{\normalfont\normalsize\bfseries\boldmath}}

\renewcommand\subsection{\@startsection{subsection}{2}{\z@}
{-3.25ex\@plus -1ex \@minus -.2ex}
{1.5ex \@plus .2ex}
{\normalfont\normalsize\bfseries\boldmath}}

\renewcommand{\@seccntformat}[1]{\csname the#1\endcsname. }

\makeatother

\newtheorem{Theorem}{Theorem}
\newtheorem{Lemma}{Lemma}
\newtheorem{Conjecture}{Conjecture}
\newtheorem{Proposition}{Proposition}
\newtheorem{Corollary}{Corollary}

\newtheorem*{rems}{Remarks} 
\newenvironment{Remarks}{\begin{rems}\normalfont}{\end{rems}}
\newtheorem*{rem}{Remark} 
\newenvironment{Remark}{\begin{rem}\normalfont}{\end{rem}}

\usepackage{amssymb, latexsym, amsmath, epsfig, amsthm} 
\usepackage{color}
\usepackage{graphicx}
\usepackage{hyperref}

\allowdisplaybreaks

\newcommand{\li}{\operatorname{Li}}

\newcommand{\height}{H}
\newcommand{\sumH}{S}

\newcommand{\vp}{\varphi}

\begin{document}
%Next comes your title and author list information in the following construct:

\begin{center}
\uppercase{\bf Prime number conjectures from \\ the Shapiro class structure}
\vskip 20pt
{\bf Hartosh Singh Bal}\\
{\smallit The Caravan, Jhandewalan Extn., New Delhi, India}\\
{\tt hartoshbal@gmail.com}\\ %(optional)
\vskip 10pt
{\bf Gaurav Bhatnagar\footnote{Current address: School of Physical Sciences,
Jawaharlal Nehru University, Delhi, India
\\
}}\\
{\smallit Fakult\"at f\"ur Mathematik,  Universit\"at Wien, 
%Oskar-Morgenstern-Platz 1, A-1090 
Vienna, Austria.}\\
{\tt bhatnagarg@gmail.com}\\ %(optional)
\end{center}
\vskip 20pt
%\centerline{\smallit Received: , Revised: , Accepted: , Published: } % We will fill in the dates
\vskip 30pt

%Followed by your abstract in the following construct:
\centerline{\bf Abstract}
\noindent
The height $H(n)$ of 
$n$, introduced by  Pillai in 1929,  is the 
smallest positive integer $i$ such that the  $i$th iterate of Euler's totient function at $n$~is~$1$. 
H.\ N.\ Shapiro (1943) studied the structure of the set of all numbers at a height.   
We state a formula for the height function due to Shapiro and use it to list steps to
generate numbers at any height. This turns out to be a useful way to think of this construct. In particular, we extend some results of Shapiro regarding the largest odd numbers at a height. We present some theoretical and computational evidence to show 
that $H$ and its relatives are closely related to the important functions of number theory, namely 
$\pi(n)$ and the $n$th prime $p_n$. We conjecture formulas for 
$\pi(n)$ and $p_n$ in terms of the height function.  \\

%\keywords{Euler's totient function, Iterated totient function, Prime numbers, Multiplicative Functions, Chebyshev's theorem, Shapiro Classes}
%\subjclass[2010]{Primary 11A25; Secondary 11A41, 11N37}

%\pagestyle{myheadings} 
%\markright{\smalltt INTEGERS: 19 (2019)\hfill} 
%\thispagestyle{empty} 
\baselineskip=12.875pt 
\vskip 30pt

\section{Introduction}\label{sec:intro} 
The principal object of our investigation is a number theoretic function $H$ that we call the {\em height function}. It is
defined as follows.  
Let $H(1):=0$, and 
\begin{equation}\label{H-def}
H(n) := H(\varphi(n))+1,
\end{equation}
for $n= 2, 3, 4, \dots$.
Here  $\vp(n)$ denotes Euler's totient function, the number of positive integers less than $n$ 
which are co-prime to $n$. 
The first few values of $H$ are: $0, 1, 2, 2, 3, 2$. 
%We will show that $H(n)$ is analogous to 
%$\log (n)$. 
%We can visualize all the natural numbers as vertices in a labelled tree. Starting from a root vertex labelled $1$, we build a tree  where for vertex labelled $n>1$, $n$ is 
%connected by an edge to $\vp(n)$. Looked this way, $\height(n)$ is the number of edges in the path from the root $1$ to $n$.  
%We can also visualize this tree as a poset, where $n$ covers 
%$\vp(n)$, for all $n>1$.
 Let $C_k$ be the set of numbers at height $k$, that is, 
$$C_k:=\{ n: \height(n)=k\}.$$  
We call $C_k$ the {\em Shapiro classes} in honor of Harold N.\ Shapiro \cite{shapiro1943} who studied these classes first. 
In this paper, we examine the Shapiro class structure, and show that the height function and its relatives are very closely 
related to the important functions of number theory, namely $\pi(n)$, the number of primes less than or equal to $n$, and $p_n$ the $n$th prime number. 
%This relationship comes from elementary considerations of the Shapiro classes. 

Shapiro arrived at these classes by considering iterates of Euler's totient function.
For $i\geq 1$, we denote by
 $$\vp^i(n):=\underbrace{\vp(\vp(\cdots))}_{i \text{ times}}(n)$$ the $i$th iterate of 
$\vp$. Then the height function can also be defined as follows. 
Let $\height(1):=0$. For $n>1$, let 
$H(n)$ be the smallest number $i$ such that $\vp^i(n)=1$. The height function (with this definition) was studied first by Pillai \cite{pillai1929}. The Shapiro classes (by other names)  have been studied earlier by Shapiro \cite{shapiro1943},  Erd\H{o}s, Granville, Pomerance and Spiro \cite{EGPS1990}, and others \cite{catlin1970,  mills1943,  noe2008, parnami1967}. We have mildly modified Shapiro's formulation.

%We consider also the average height function which we call the {\em level function} defined as:
%\begin{equation}\label{level-def}
%\level(n):= \frac{1}{n} \sum_{k=1}^n \height(n).
%\end{equation}
%

In Section~\ref{sec:formula} we find an alternative, inductive, approach to 
generate Shapiro classes. To do so, we
require a formula for $H(n)$ due to
Shapiro. 
%This 
%approach allows us to recover many properties of Shapiro classes easily. 
 In Section~\ref{sec:class-structure}, we apply our ideas to extend a theorem of Shapiro about the 
largest odd numbers at a given height. 
This rests upon a property that is obvious from our construction, but not observed earlier, regarding the largest possible prime number at a height. 
This, and results of Erd\H{o}s et.al.~\cite{EGPS1990}, led us to look for number theoretic information from this structure. 
 
Since $H(n)$ is not an increasing function, we consider  the sum of heights function,  defined as: $S(0):=0$, and 
\begin{equation}\label{sumH-def}
\sumH(n):= \sum_{k=1}^n \height(k).
\end{equation}
%While the average of $H$ given by $S(n)/n$ is still not an 
%increasing function, it is still better behaved than the height function. 
The structure consisting of Shapiro classes allows  
us to obtain number theoretic information quite easily. 
It appears that elementary techniques, such as those found in the textbooks of Apostol \cite{apostol1976} and Shapiro \cite{shapiro1983}, can be modified to express classical theorems
in terms of functions related to the height function. 
We illustrate this idea in Section~\ref{sec:chebyshev}, by proving Chebyshev-type theorems, that is, inequalities for $\pi(n)$ and $p_n$ in terms of $n$ and $S(n)$. 

The function $S(n)$ appears to be important to the theory of prime numbers. 
Consider, for example, the values of $S(n)$ for $n=1, 2, \dots, 9$. These are
\begin{align*}
0, &1, 3, 5, 8, 10, 13, 16, 19.\cr
\intertext{Compare these with the first few prime numbers:}
&2, 3, 5, 7, 11, 13, 17, 19.
\end{align*}
%It appears that $p_n$ is approximated by $S(n)$. 
This pattern persists; consider Figure~\ref{fig:nthprime}, a plot of $S(n)$ and $p_n$ on the same set of axes. There is remarkable agreement of these graphs (up to $n=5000$). From numerical computations
(see Tables~\ref{table:S-nthprime1} and \ref{table:S-nthprime2}), it appears that $S(n)$ is a better approximation to $p_n$ than $n\log n$, at-least until $n=6\times10^7$.
\begin{figure}[h]
\begin{center}
\includegraphics[scale=0.50]{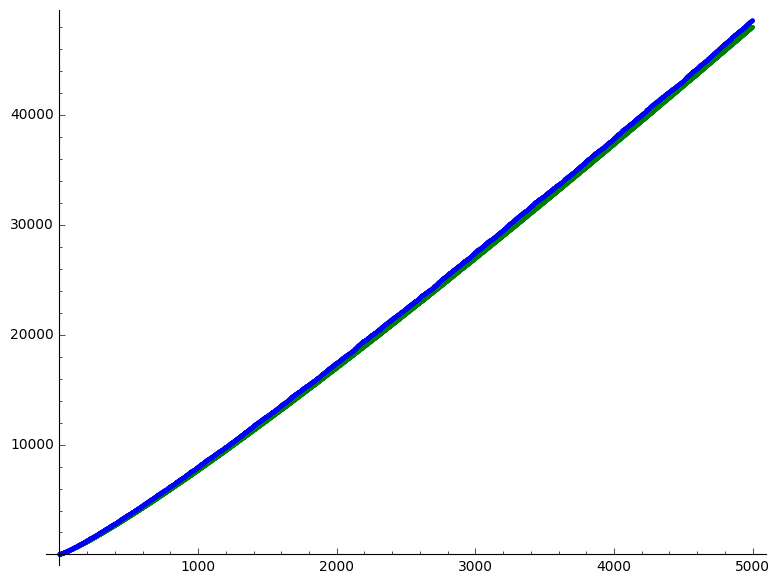}
\caption{Comparison of the $n$th prime (top) with $S(n)=\sum_{k=1}^n H(k)$. }
\end{center}
\label{fig:nthprime}
\end{figure}
\begin{table}[h]
$$
\begin{array}{| c | c | c | c | c | c  |c|} \hline
n& p_n & S(n) & S(n)-p_n & \lfloor n\log n \rfloor & \lfloor n\log n \rfloor -p_n & \frac{S(n)-p_n}{p_n} \\
\hline
10 & 29 & 22 & -7 & 23 & -6 & -24.1 \% \cr \hline
10^2 & 541 & 486 & -55 & 460 & -81 & -10.2 \% \cr \hline
10^3 & 7919 & 7640 & -279 & 6907 & -1012 & -3.52 \% \cr \hline
10^4 & 104729 & 104488 & -241 & 92103 & -12626 & -0.230 \% \cr \hline
10^5 & 1299709 & 1325890 & 26181 & 1151292 & -148417 & 2.01 \% \cr \hline
10^6 & 15485863 & 16069024 & 583161 & 13815510 & -1670353 & 3.77 \% \cr \hline
10^7 & 179424673 & 188786066 & 9361393 & 161180956 & -18243717 & 5.22 \% \cr \hline
\end{array}
$$
\caption{$S(n)$ against $ \lfloor n\log n \rfloor$}
\label{table:S-nthprime1}
\end{table}

This ties up with a result/conjecture of 
Erd\H{o}s, Granville, Pomerance and 
Spiro~\cite{EGPS1990} and 
motivates the experimental work presented in Section~\ref{sec:conjectures}.

%Our thesis in this paper is that the function $\sumH(n)$ is a very natural function for number theoretic 
%studies. Note that $$\height(2^h)=h=\log_2(2^h).$$
%Further, as Shapiro showed
%$$\height(p\cdot q) = \height(p)+\height(q)-1,$$
%for odd primes $p$ and $q$, which is to be compared with
%$$\log_2(p\cdot q) = \log_2(p)+\log_2(q).$$
%(We note that $\height(2p)=\height(p)$ for $p$ an odd prime.)
%Intuitively $\height(n)$ is a kind of number theoretic $\log$ function. But it does not increase with 
%$n$. 
%So we would like to average it out and define the sum $S(n)$ as above. We will show that this function behaves 
%like $n\log n$. We also claim that $S(n)/n$ can replace $\log n$ in many number theoretic contexts and often gives better estimates than the classical ones, especially for small numbers. 
%

%The textbooks we use for referring to background information are Apostol \cite{apostol1976} and 
%Shapiro \cite{shapiro1983}.

\section{Listing Shapiro classes}\label{sec:formula}
We begin our study  of the height function and Shapiro classes. The objective of this section to build some useful intuition, by writing a set of rules to generate Shapiro classes inductively. 
Towards this end, we first prove a formula due to Shapiro for calculating the height of a function. 
As a corollary, we show 
the additive nature of the height function.
We illustrate the 
usefulness of these rules by obtaining several elementary properties of the Shapiro classes. 

It is instructive to compute the first few classes.  Table~\ref{table:first-3-rows} gives the first four Shapiro classes. We begin with $1$ in $C_0$ and $2$ in $C_1$. Now since $\varphi(3)=2$, we find that $H(3)=H(2)+1=2$. So $3\in C_2$. Similarly,  we see that $4$ has height $2$, $5$ has height $3$, and $6$ has height $2$. 
\begin{table}[h]
$$
\begin{array}{c|| l }
\hline
k & C_k \cr \hline
3 & 5, 7, 8, 9, 10, 12, 14, 18 \cr \hline
2 & 3, 4, 6 \cr \hline
1 & 2 \cr \hline
0 & 1 \cr \hline
\end{array}
$$
\caption{$C_k$ for $k=0, 1, 2, 3$.}
\label{table:first-3-rows}
\end{table}
By computing the values of $H(n)$ for a few more values, a few rules for finding members of the Shapiro class become evident. For example, consider the following rules for finding members of $C_k$ from $C_{k-1}$. We have
\begin{itemize}
\item If $m$ is an odd number, then the height of $2m$ is the same as $m$. 
\item If $m$ is an even number, then $H(2m)=H(m)+1$. 
\item If $m$ is an odd number, then $H(3m)=H(m)+1$. 
\item If $p$ is a prime number, then $H(p)=H(p-1)+1$. Thus $p-1$ is an even number at one height lower than $p$. 
\end{itemize}
We can use these rules to generate the members of $C_4$. If we multiply all the even numbers of 
$C_3$ by $2$, we obtain $16$, $20$, $24$, $28$, $36$. Now multiplying all the odd numbers by 
$3$, we find that $15$, $21$, $27$ are in $C_4$. Next, consider  $8+1=9$, $10+1=11$, 
$12+1=13$, $14+1=15$, $18+1=19$. Of these, $11$, $13$, and $19$, are primes, and are thus at 
height $4$. Finally, the odd numbers already obtained are: $11, 13, 15, 19, 21, 27$. Multiplying 
them by $2$, we find that $22, 26, 30, 38, 42,$ and $54$ are also in $C_4$. The reader may verify 
that we have obtained all the numbers at height $4$. 
%by comparing with Table~\ref{classes}.

The rules work to generate all the elements of $C_4$ from $C_3$, but they are not comprehensive. The complete set of rules will appear shortly, as an application of a formula for $H(n)$. 
\begin{Theorem}[Shapiro \cite{shapiro1943}]\label{th:hn-formula} 
Let $n>1$ with prime factorization $n=2^\alpha p_1^{\alpha_1}p_2^{\alpha_2}\dots p_r^{\alpha_r}$, where 
 $p_1$, $p_2$, $p_3$, $\dots$, $p_r$ are distinct odd primes, and $\alpha, \alpha_1$, 
 $\alpha_2,$ $\dots,$ $\alpha_r\geq 0$. Then,
\begin{equation}\label{H}
H(n)=
\begin{cases}
 \alpha+\sum\limits_{i=1}^r \alpha_i (H(p_i)-1)& \text{ if } \alpha>0
 \cr
 \sum\limits_{i=1}^r \alpha_i (H(p_i)-1) +1 & \text{ if } \alpha =0.
 \end{cases}
 \end{equation}
\end{Theorem}
%\begin{Remark} Let $n>1$ be as in Theorem~\ref{th:hn-formula}. 
%Formula \eqref{H} is analogous to the expression
%$$\log n = \alpha\log 2 + \sum_{i=1}^r \alpha_i \log p_i.$$
%\end{Remark}
Before giving a proof of this theorem, we obtain a corollary which indicates the additive nature of  the height function. 
\begin{Corollary}[Shapiro \cite{shapiro1943}]\label{cor:sub-additivity} Let $n$ and $m$ be natural numbers with prime factorizations
$m=2^\alpha p_1^{\alpha_1}p_2^{\alpha_2}\dots p_r^{\alpha_r}$  and 
$n=2^\beta p_1^{\beta_1}p_2^{\beta_2}\dots p_r^{\beta_r},$
%\begin{gather*}
%m=2^\alpha p_1^{\alpha_1}p_2^{\alpha_2}\dots p_r^{\alpha_r},\cr
%\intertext{and}
%n=2^\beta p_1^{\beta_1}p_2^{\beta_2}\dots p_r^{\beta_r},
%\end{gather*}
where $\alpha, \beta, \alpha_i, \beta_i \geq 0$. Then we have
\begin{equation}
H(nm) = 
\begin{cases}
H(m)+H(n),  & \text{ if } \alpha, \beta>0; \cr
H(m) + H(n)-1, & \text{ if  }  \alpha =0 \text{ or } \beta = 0 .
\end{cases}
\end{equation}
\end{Corollary}
\begin{Remark} Shapiro \cite{shapiro1943} considered Corollary~\ref{cor:sub-additivity} as his fundamental theorem. The formula \eqref{H} follows immediately. The proof presented below is closely related to Shapiro's proof of Corollary~\ref{cor:sub-additivity}. 
\end{Remark}
\begin{proof}
Note that if both $\alpha$ and $\beta$ are strictly positive, then
\begin{align*}
H(mn) &= \alpha +\beta+ \sum_{i=1}^r (\alpha_i+\beta_i)\left( H(p_i)-1\right) \cr
&= H(n)+H(m),
\end{align*}
as required.
Similarly, the second part of the formula follows by considering the cases where only one of $\alpha$ and $\beta$ is $0$, and where both $\alpha$ and $\beta$ are $0$. 
\end{proof}
\begin{proof}[Proof of Theorem \ref{th:hn-formula}]
We prove this formula by induction on the number of primes in the prime factorization of $n$. Remarkably, we will require the computation in the corollary to prove the theorem. 

First we prove the formula for numbers $n$ of the form 
$n=2^\alpha$, by induction. Clearly, the formula holds for $\alpha=1$. Suppose it holds for $\alpha =k$. Then 
$$H(2^{k+1})=H\big(\vp(2^{k+1})\big)+1 = H(2^k)+1 = k+1$$
by the induction hypothesis. This proves the formula for powers of $2$. 

Next, we consider numbers $n$ of the form $n=2^\alpha\cdot3^{\alpha_1}$. Consider first the case $\alpha>0$. Again we will prove this formula by induction, this time on $\alpha_1$. 
It is easy to verify the formula for $\alpha_1=1$. 
%Observe that $\vp(3)=2$, so $H(3)=H(2)+1=2.$ When $\alpha_1=1$, we can verify that
%$$H(2^\alpha\cdot 3) = H\big(\vp(2^\alpha\cdot 3)\big) +1 = H(2^\alpha)+1 = \alpha+1\cdot(H(3)-1).
%$$
 Next suppose the formula is true for $\alpha_1=k$. For $n=2^\alpha\cdot 3^{k+1}$, we have
$$H(2^\alpha\cdot 3^{k+1}) = H\big(\vp(2^\alpha\cdot 3^{k+1})\big) +1 
= H(2^\alpha\cdot 3^k)+1 = \alpha+(k+1)\cdot(H(3)-1)
$$
as required. 

Next we consider the case $\alpha=0$.  In this case, we can obtain the formula from the previous case. Since $\alpha_1\geq 1$, We have
$$H(n)=H\big(\vp(n)\big) +1= H\big(2\cdot 3^{\alpha_1-1}\big)+(H(3)-1)  = 1+\alpha_1(H(3)-1)$$
which is the required formula. 

Next, as the induction hypothesis, we assume that the formula is valid for numbers 
%$n$ 
of the form
%$$n=
$$2^\alpha p_1^{\alpha_1}\cdots p_r^{\alpha_r},$$
where $p_1<p_2<\cdots < p_r$ are the first $r$ odd primes. We will show the formula for 
$$n = 2^\alpha p_1^{\alpha_1}\cdots p_r^{\alpha_r} p_{r+1}^{\alpha_{r+1}}.$$ 
We first consider the case $\alpha>0$. Again we will prove this by induction on $\alpha_{r+1}$.
Let 
\begin{align*}
n^{\prime} &= \vp(p_{r+1}) = 2^\beta p_1^{\beta_1} \cdots p_r^{\beta_r}\cr
m^{\prime} &= \vp(n/p_{r+1}^{\alpha_{r+1}}) = 2^\gamma p_1^{\gamma_1} \cdots p_r^{\gamma_r}.
\end{align*}
We will first prove the formula for $\alpha_{r+1}=1$. 
Then, 
$$\vp(n) = \vp(p_{r+1}) \vp(n/p_{r+1}),$$
so
\begin{align*}
H(n)&=H(\vp(n))+1 \cr
&=H(m^{\prime}n^{\prime})+1\cr
&= H(m^{\prime}) + H(n^{\prime}) +1 .
\end{align*}
The last equality follows by the argument in Corollary~\ref{cor:sub-additivity} applied to even 
numbers that have  
only $2$, $p_1$, $\dots$, $p_r$ in their prime factorization. (By the induction hypothesis, we have 
assumed that \eqref{H} holds for such numbers, 
and so this argument can be used in our proof.) 
Thus we find that 
\begin{align*}
H(n) &= H(m^{\prime}) + H(n^{\prime})+1 \cr
&= H\big( 2^\alpha p_1^{\alpha_1}\cdots p_r^{\alpha_r} \big) -1 + H(p_{r+1})-1 +1\cr
&= \alpha +\sum_{i=1}^r \alpha_i \big( H(p_i)-1\big) +H(p_{r+1}) -1,
\end{align*}
as required. 

Next, we suppose the formula is true for $\alpha_{r+1}=k$. We will prove it for 
$$n= 2^\alpha p_1^{\alpha_1}\cdots p_r^{\alpha_r} p_{r+1}^{k+1}.$$
Again, let $m^\prime$, $n^\prime$, $\beta_i$, $\gamma_i$ be as before. Then we have
\begin{align*}
H(n) & = H(\vp(n)) +1 \cr
&= H\big( 2^{\beta+\gamma} p_1^{\beta_1+\gamma_1}\cdots p_r^{\beta_r+\gamma_r}
p_{r+1}^{k}  \big) +1 \cr
&= \beta+\gamma + \sum_{i=1}^r (\beta_i+\gamma_i)\big( H(p_i)-1\big) + k(H(p_{r+1}-1)) +1\cr
&= H(\vp(p_{r+1})) +  H\big(\vp(2^\alpha p_1^{\alpha_1}\cdots p_r^{\alpha_r})\big) + 
    k(H(p_{r+1})-1) +1\cr
 & = H\big(2^\alpha p_1^{\alpha_1}\cdots p_r^{\alpha_r}\big) + (k+1)(H(p_{r+1})-1),
\end{align*}
which proves the formula for $\alpha>0$. 

Finally, we prove the formula for $\alpha=0$ and $\alpha_{r+1}\geq 1$. Let 
$$n = p_1^{\alpha_1}\cdots p_r^{\alpha_r} p_{r+1}^{\alpha_{r+1}}.$$
Now 
\begin{align*}
H(n) &  = H(\vp(n)) +1 \cr
& = H\big( 2^{\beta+\gamma} p_1^{\beta_1+\gamma_1} \cdots p_r^{\beta_r+\gamma_r}
p_{r+1}^{\alpha_{r+1}-1}  \big) +1\cr
& = H(p_1^{\alpha_1}\cdots p_r^{\alpha_r}) + \alpha_{r+1}(H(p_{r+1})-1)+1,
\end{align*}
where the last step is obtained as in the $\alpha>0$ case. From here we obtain the 
formula by applying the induction hypothesis. 
\end{proof}

We return to our set of rules to generate the numbers at a given height from the previous ones. 
%Recall the notation $C_k$ for the set of numbers at height $k$. 
Corollary~\ref{cor:sub-additivity} tells us what happens if we multiply a number by $5$. Recall that the height of $5$ is  $3$. So if $m$ is any number, then 
$H(5m)=H(m)+2$. Thus to obtain numbers at height $k$ we have to multiply numbers at height $k-2$ by $5$ and the other prime at height $3$, namely $7$. In general, to obtain all the numbers at a height, we have to consider all the primes at lower heights. We are now in a position to create
a comprehensive list of rules to generate $C_k$.  

Let $Q_k$ denote the set of primes at height $k$. That is, for $k=1, 2, 3, \dots$, define
$$Q_k:= \{ p :  p  \text{ a prime and } H(p)=k\}.$$ 
For example: 
$$Q_1=\{ 2 \}, Q_2=\{ 3 \}, Q_3 = \{ 5, 7 \}, Q_4=\{ 11, 13, 19\}, 
Q_5=\{ 17, 23, 29, 31, 37, 43\}.$$ 
\begin{Theorem}\label{th:algo1} Suppose $C_1$, $C_2$, $\dots,$ $C_{k-1}$ are known.
The steps for obtaining $C_k$, that is, the elements at height $k$ are as follows
\begin{enumerate}
\item Multiply each even element of  $C_{k-1}$ by $2$.
\item Multiply each odd element of $C_{k-1}$ by $3$.
%\item Multiply each element of height $C_{k-2}$ by primes at height $3$, i.e., by elements of $Q_3$. 
\item 
Multiply each element of $C_{k-k_1}$ by elements of $Q_{k_1+1}$, where $k_1=2, 3, \dots.$
\item If $m+1$ is a prime where $m$ is an even 
number in $C_{k-1}$, then $m+1\in C_k$. 
\item Multiply each odd number obtained already in $C_k$ by $2$. 
\end{enumerate}
\end{Theorem}
\begin{Remark}
Shapiro considered a slightly different function $C(n)$, where $C(1)=0$ and 
$C(n)=H(n)-1$, for  $n>1$. So $C(2)=0$ whereas $H(2)=1$.  Theorem~\ref{th:algo1} is the primary reason why we have deviated from Shapiro's formulation. 
\end{Remark}
\begin{proof} 
The proof is a formalization of our earlier discussion. Consider the following statements.
\begin{enumerate}
\item If $x\in C_{k-1}$ is even, then $H(2x) = k$. 
\item if $x\in C_{k-1}$, then $H(3x)=k$.
\item if $x\in C_{k-k_1}$ and $p\in Q_{k_1+1}$ then 
$H(px)
 %=(k_1+1)+(k-k_1)-1 
= k$. 
\end{enumerate}
All these statements follow from Corollary~\ref{cor:sub-additivity}.

Step 1 generates all even numbers of the form $2^a m$ where $a>1$, and $m$ is odd, i.e., numbers divisible by $4$. The steps 2 and 3 will generate all composite, odd, numbers in  $C_k$. 

Since $\phi(p)=p-1$, and for all odd primes $p-1$ is even, thus all prime numbers are obtained by adding $1$ to even numbers in $C_{k-1}$. This explains Step 4. 

It remains to generate numbers of the form $2m$ where $m$ is odd. Since $H(2n)=H(n)$ when $n$ is odd, we must multiply each odd number obtained by steps 1--4 by $2$. This will generate all the elements of $C_k$. 
\end{proof}
\begin{Remark} In Step 3, we need to multiply only those elements of $C_{k-k_1}$ that 
are odd and not divisible by $3$.  
\end{Remark}

Now it is easy to generate some more sets $C_k$. The elements of the first few sets $C_k$ are given in Table~\ref{classes}. The prime numbers are
given in bold.
\begin{table}[h]
\begin{center}
\begin{tabular}[h]{| l | l | }
\hline\hline
$k$ & $C_k$\cr
\hline
$6$ & $\bold{41}, \bold{47}, 51, \bold{53}, 55, \bold{59}, \bold{61}, 64, 65, \bold{67}, 68, 69, 
\bold{71}, \bold{73}, 75, 77,$\\
& $\bold{79}, 80, 82, 87, 88, 91, 92, 93, 94, 95, 96, 99, 100, 102, 104, $\\
 & $ 105, 106, \bold{109}, 110, 111, 112, 116, 117, 118, 120, 122, 124, $\\
 & $ \bold{127}, 129, 130, 132, 133, 134, 135, 138, 140, 142, 144, 146,$ \\
& $  147, 148, 150, 152, 154, 156, 158, \bold{163}, 168,  171, 172, 174,$\\
& $ 180, 182, 186, 189, 190, 196, 198, 210, 216, 218, 222, 228,$\\
& $  234, 243, 252, 254, 258, 266, 270, 294, 324, 326, 342, 378,$\\
& $  486 $\cr
$5$ & $\bold{17}, \bold{23}, 25, \bold{29}, \bold{31}, 32, 33, 34, 35, \bold{37}, 39, 40,
\bold{43}, 44, 45, 46,$ \\
&$ 48, 49, 50, 52, 56, 57, 58, 60, 62, 63, 66,
70, 72, 74, 76, 78, $\\
&$ 81, 84, 86, 90, 98, 108, 114, 126, 162$\cr
$4$ & $\bold{11}, \bold{13}, 15, 16, \bold{19}, 20, 21, 22, 24, 26, 27, 28, 30, 36, 38, 42, 54$ \cr
%\hline
$3$ & $ \bold{5}, \bold{7}, 8, 9, 10, 12, 14, 18$ \cr
%\hline
$2$ & $\bold{3}, 4, 6 $ \cr
%\hline
$1$ & $ \bold{2} $ \cr
%\hline
$0$ & $1$ \cr
\hline
\end{tabular}
\end{center}
\caption{Numbers with height $\le 6$. The primes are in bold.}\label{classes}
\end{table}

Theorems~~\ref{th:hn-formula} and \ref{th:algo1} are very useful to think about the $C_k$. We illustrate this in the following observations.   These have been found previously by
other authors. 
%In the remainder of the section, we note these results which will be useful in the rest of the paper. 
\subsection*{Observations}
\begin{enumerate}
\item The number $2^k$ comes at height $k$. 

\item The smallest even number at height $k$ is $2^k$. To see this, consider the following argument.

An even number in $C_k$ can arise in two ways. If it is obtained by multiplying an element of $C_{k-1}$ by $2$, it is bigger than or equal to $2^k$ by induction. The other possibility is that it is of the form $2m$, where $m$ is an odd number at height $k$. 

If $m$ is a prime number, then it is obtained by adding $1$ to an even number at height $k-1$, so it is bigger than $2^{k-1}$ by induction. This implies $2m >2^k$. 

Suppose $m$ is an odd, composite number with prime factorization 
$$m=p_1^{\alpha_1}\dots p_r^{\alpha_2},$$ with $H(p_i)=k_i+1$.
Then by Theorem~\ref{th:hn-formula}, $H(m)=k$ implies that $$k-1=\sum_{i=1}^r \alpha_i k_r .$$
But by induction, as above, we must have $p_i> 2^{k_i}$ for $i=1, 2, \dots, r$. Thus 
$$m = \prod_i p_i^{\alpha_i} > 2^{\sum_i \alpha_i k_i} = 2^{k-1}.$$ 
Again, this implies that $2m>2^k$. 

\item Any odd number in $C_k$ is bigger than $2^{k-1}$.  This follows from the above argument. 

\item\label{obs:hnlen}  If $n\le 2^k$, then $H(n)\le k$. This is what the two items above amount to; this was implicit in Pillai~\cite{pillai1929}, and stated by Shapiro \cite{shapiro1943}.

\item (Shapiro \cite{shapiro1943})  The numbers at height $h$ that are less than $2^h$ are all odd.

%In other words, no number smaller than $2^{k-1}$ has height $k$. 
%
%The proof is by induction. Suppose this is true for height $k-1$ and below. Suppose for example, any 
%number at height $k$ is obtained by Step (3) in Theorem~\ref{th:algo1}. Any number at height $C_{k-k_0}$ is bigger than $2^{k-k_0-1}$ and any prime in $Q_{k_0+1}$ is bigger than $2^{k_0}$, so their product is bigger than $2^{k-1}$. A similar argument applies to any number obtained by Steps (1) and (2). As for any primes at height $k$, they are obtained by adding $1$ to an even number at height $k-1$. But the smallest even number at height $k-1$ is bigger than $2^{k-2}$ and is thus $2^{k-1}$. So any prime at height $k$ is bigger than $2^{k-1}$. 
%

%We will use this observation many times. We have written it in the form of a  Lemma and provide a useful extension, see Lemma~\ref{lem:primepowers}.

%Other useful observations that follow from Theorem~\ref{th:algo1} are as follows. 
%\begin{enumerate}
\item (Shapiro \cite{shapiro1943}) The largest odd number at height $h$ is $3^{h-1}$. This follows from induction and Theorem~\ref{th:algo1}. 
\item (Pillai~\cite{pillai1929}) The largest even number at height $h$ is $2\cdot 3^{h-1}$.
\item (Pillai \cite{pillai1929}) 
Since any number at height $h$ is between $2^{h-1}$ and 
$2\cdot 3^{h-1}$, we have the inequalities:
\begin{equation}\label{pillai} 
\frac{\log n/2}{\log 3}+1 \le H(n) \le \frac{\log n}{\log 2}+1.
\end{equation}

\end{enumerate}

All the observations above were noted by previous authors. One can ask whether 
Theorem~\ref{th:algo1}  gives any new information. 
Indeed, there is one very important observation missed by previous authors.

The largest prime at a level $k$ is less than or equal to $2\cdot 3^{k-2}+1$. This is obvious from Step~4 of Theorem~\ref{th:algo1} and Pillai's observation that the largest even number at height $k-1$ is $2\cdot 3^{k-2}$.

In the next section, we show how this observation can be used to obtain information about the numbers appearing at the end of each class, thus extending some of Shapiro's results.

\section{On the Shapiro Class structure}\label{sec:class-structure}

The objective of this section is to illustrate the application of Theorems~\ref{th:hn-formula} and 
\ref{th:algo1}, by extending some results of Shapiro \cite{shapiro1943}. Our theorem in this section is a characterization of the last few numbers at each height. 

As noted above, the largest prime at a level $k$ is less than or equal to $2\cdot 3^{k-2}+1$. This upper bound is met for many $k$. The smallest such examples are obtained when 
$k= 2,$ $3,$ $4,$ $6,$ $7,$ $8,$ $11,$ $18,$ $19$. The first few examples, corresponding to these values of $k$, are 
$ 3, 7, 19, 163, 487, 1458, 39367, 86093443, 258280327$.
Primes of this kind (cf. OEIS \cite[A003306]{sloane}) play an important role in our theorem. 
Let $\widehat{P}$ denote the set of primes of this form, that is,
$$\widehat{P}:=\{ p : p \text{ a prime}, p = 2\cdot 3^{k-2} +1 \text{ for some } k\ge 2 \}.$$

To prove our main result, we require a useful proposition. 
\begin{Proposition}\label{prop:boundforodd}
 Let $k>2$. Let $m$ be an odd, composite number not divisible by $3$ at height $k$. Then,
$$m< 2\cdot 3^{k-2}+1.$$
\end{Proposition}
Before proving the proposition, we prove a special case, where $m$ is of the form $pq$ or $p^2$. 
\begin{Lemma} \label{lem:pqbound}
Let $k>2$. Let $p$ and $q$ be (possibly the same) primes. Suppose $p, q\neq 2, 3$, and $H(pq)=k$. Then
$$pq< 2\cdot 3^{k-2}+1.$$
\end{Lemma}
\begin{proof} Let $H(p)=a$ and $H(q)=b$. Since $p$ and $q$ are not $2, 3$, we must have $a,b >2$. Since $p \le 2\cdot 3^{a-2} +1$ and $q \le 2\cdot 3^{b-2} +1$,
%\begin{align*}
%p &\le 2\cdot 3^{a-2} +1 \\
%\intertext{and,}
%q &\le 2\cdot 3^{b-2} +1,
%\end{align*}
we must have
$$pq \le 4\cdot 3^{a+b-4} + 2\cdot 3^{a-2} + 2\cdot 3^{b-2} +1.$$
Now from Corollary~\ref{cor:sub-additivity}, we see that $k=H(pq)=a+b-1$. Now consider 
\begin{align*}
2\cdot 3^{a+b-3}+1 & = 6\cdot 3^{a+b-4} +1 \cr
&= 4\cdot 3^{a+b-4} + 3^{b-2} \cdot 3^{a-2} + 3^{a-2} \cdot 3^{b-2} +1 \cr
& > 4\cdot 3^{a+b-4} + 2\cdot 3^{a-2} + 2\cdot 3^{b-2} +1,
\end{align*}
since $a>2$ implies that $3^{a-2}>2$.
This completes the proof of the lemma. 
\end{proof}
\begin{proof}[Proof of Proposition~\ref{prop:boundforodd}]
We use induction on $k$. For $k=3, 4$, the statement is vacuously true. Let $k\geq 5$. Let $p|m$, $H(p)=a$ and $H(m/p)=b$. Then $a,b >2$, otherwise $m$ has to be divisible by $3$. Further $b<k$ since $a+b-1=k$. Thus by the induction hypothesis, we must have 
$$m/p < 2\cdot 3^{b-2}+1.$$
Of course, since $p$ is a prime,
$p<2\cdot 3^{a-2}+1.$
Now by using the same argument as in Lemma~\ref{lem:pqbound}, we see that
$$m< 2\cdot 3^{a+b-3}+1 =  2\cdot 3^{k-2}+1,$$
as required. 
\end{proof}

Next, we determine all the numbers in the set
$$R_k := \{ n \in C_k : 2^2 3^{k-2}<n\le 2\cdot 3^{k-1}\}.$$ 
These numbers are the largest elements at height $k$. 

\begin{Theorem} Let $k>2$. The set $R_k$ comprises all the numbers of the form $m$, where
$$m = 2\cdot 3^{k-a}p,$$
where $p\in \widehat{P}$, with $H(p)=a$ and $a\le k$. 
\end{Theorem}

\begin{proof}
We let $S_k$ denote the set
$$S_k:= 
\{ m: m =  2\cdot 3^{k-a}p \text{ for some } p \in \widehat{P}, \text{ with } H(p)=a \text{ and } a\le k \}.
$$
We want to show that $R_k=S_k$. 

First observe from \eqref{H} that if $m\in S_k$, then $H(m)=k$. Further, if $m \in S_k$, then
$$2^23^{k-2}< m \le 2\cdot 3^{k-1}.$$
This follows from
\begin{equation}\label{sk-ordering}
2\cdot 3^{k-a}p = 2\cdot 3^{k-1}\Big( \frac{2}{3}+\frac{1}{3^{a-1}}\Big).
\end{equation}
Thus $S_k\subset R_k$. 

To show the converse, we apply  an inductive argument using Theorem~\ref{th:algo1}.

For $k=3$, $R_3=\{ 14, 18 \} = S_3.$ So let $k>3$. 

Observe that $R_k$ has only even numbers. This is because all odd numbers at height $k$ are less than or equal to $3^{k-1}$, and $3^{k-1}< 4\cdot 3^{k-2}$. 

Even numbers are obtained from Step 1 or Step 5 in Theorem~\ref{th:algo1}. However, since there are no numbers in $C_{k-1}$ that are bigger than $2\cdot 3^{k-2}$, none of the numbers in $R_k$ are obtained from Step 1. Thus all the numbers in $R_k$ are of the form $2r$, where $r$ is an odd number at height $k$, and 
$$2\cdot 3^{k-2}< r \le 3^{k-1}.$$
By Proposition~\ref{prop:boundforodd}, all odd composite numbers not divisible by $3$ are less than 
$2\cdot 3^{k-2}+1$. Thus there are only two possibilities for $r$.
\begin{enumerate}
\item $r$ is a prime of the form $2\cdot 3^{k-2}+1$ with $H(r)=k$, i.e., $2r\in S_k$ as required. 
\item $r$ is divisible by $3$.
\end{enumerate}
In case $r$ is divisible by $3$, it is obtained from Step 2, and there is an $s$ such that
$2\cdot 3^{k-3} < s \le 3^{k-2}$, with $r=3s$. So $2s\in R_{k-1}$. By the induction hypothesis, 
$2s=2\cdot 3^{k-1-a}p$, for some $p\in \widehat{P}$ with $H(p) = a \le k-1$. So 
$$2r =  2\cdot 3^{k-a}p,$$
and $2r\in S_k$. This completes the proof of $S_k\subset R_k$. 
\end{proof}

To state our next result, we require some notation. 
Let $p_{1},$ $p_{2},$ $p_{3},\dots$ be the elements of $\widehat{P}$ listed in increasing order, where   $H(p_i)=a_i$, 
and $a_1< a_2 < a_3 < \cdots$. For example, the first few pairs $(p_i,a_i)$ are
$(3,2)$, $(7,3)$, $(19,4)$, and $(163,6)$. 

\begin{Corollary} Let $k>2$. Let $p_{i}\in \widehat{P}$, $i=1, 2, \dots, r$, be as above, with $r$ the largest such that
$a_r\le k$.  At height $k$, the largest numbers, in decreasing order, are:
$$2\cdot3^{k-1}, 2\cdot 3^{k-3}\cdot 7,  2\cdot 3^{k-4}\cdot 19, 2\cdot 3^{k-6}\cdot 163
, \dots, 2\cdot 3^{k-a_r}\cdot p_{r}, 2^2\cdot 3^{k-2}.$$

\end{Corollary}
\begin{proof}
Let $m_i=2\cdot 3^{k-a_r}p_{i}$.
Note that from \eqref{sk-ordering}, it follows that if $a_i > a_j$, then $m_ i< m_j$. 
Thus $2\cdot 3^{k-a_i}p_{i}$, $i=1, 2, \dots, r$, are in decreasing order. 
\end{proof}
This immediately implies a similar result for the largest odd numbers at a height. 
\begin{Corollary}\label{cor:shapiro-plus}  Let $k>2$. Let $p_{i}\in \widehat{P}$, $i=1, 2, \dots, r$, be as above, with $r$ the largest such that
$a_r\le k$.   At height $k$, the largest odd numbers, in decreasing order, are:
$$3^{k-1}, 3^{k-3}\cdot 7,  3^{k-4}\cdot 19, 3^{k-6}\cdot 163, 3^{k-a_5}\cdot p_{5},  \dots, 3^{k-a_r}\cdot p_{r}.$$
\end{Corollary}

\begin{Remark} Corollary~\ref{cor:shapiro-plus} extends Theorems 10, 11 and their corollary  from
Shapiro \cite[\S5]{shapiro1943}. Shapiro considers only two primes in $\widehat{P}$, namely
$7$ and $19$. 
\end{Remark}

To summarize our work so far, we have found in Theorem~\ref{th:algo1}
an alternative way of thinking about the Shapiro classes. We saw above how a rather obvious 
observation, about the largest possible prime in a class, can be used to obtain more information 
about  the numbers that appear at a height. 

At this point, we would like to venture a comment of a philosophical nature, motivated by another innocuous observation about primes in Shapiro classes.

Theorem~\ref{th:hn-formula} suggests that $H(n)$ is a ``measure of complexity'' of a number. The prime numbers can be considered the
``atoms'' of numbers. A number is built from $1$ by successive multiplication by prime 
numbers, so the number of prime powers dividing a number says something about how complicated a number is.  % But what about the prime numbers themselves? 
However, this construction does not distinguish between two primes. 
On the other hand, the Shapiro class structure naturally distinguishes between the primes.
On looking at Table~\ref{classes}, we see that 
primes don't come in order. For example, $19$ appears at height $4$ and $17$ at height $5$. 
% numbers appear 
%naturally. 
Thus, the height function gives a ``measure of complexity'' to each prime, and indeed, to each number. 
%Our formula for 
%$H(n)$ in Theorem~\ref{th:hn-formula} gives due importance to the
%prime powers dividing $n$. 
%Thus any number $n$ at height $k$ (i.e., with $H(n)=k$) can be obtained from $1$ by $k$ or 
%$k+1$ operations. 
%Theorem~\ref{th:algo1} describes precisely what these operations are. 
That is %why we consider $H(n)$ to be a measure of the complexity of numbers and 
why we expect this construct will say something about prime numbers. 

\section{Chebyshev-type theorems}\label{sec:chebyshev}
In this section, we explore one strategy to discover what Shapiro classes imply for prime numbers.  The strategy is to study elementary methods explained in 
Shapiro \cite[Chapter 9]{shapiro1983} and Apostol \cite[Chapter 4]{apostol1976}, and express classical results using $H(n)$ and $S(n)$. The objective of this preliminary investigation is to arrive at suitable functions that are related to the prime number functions $\pi(n)$ and $p_n$. 

We derive results analogous to Chebyshev's theorem,  which states that there are constants $0<a<A$ such that, for $n>1$, 
\begin{equation}\label{chebyshev}
a \frac{n}{\log n} <\pi(n) < A \frac{n}{\log n}.
\end{equation}
According to Apostol \cite[Theorem 4.6, (14) and (18)]{apostol1976}, we can take $a=1/4$ and $A=6$, when $n$ is an even number. 

We will use this result to provide an alternate formulation of
Chebyshev's theorem in terms of $S(n)$. In addition, we  
find inequalities for $p_n$ by modifying the proof of Apostol \cite[Theorem 4.7]{apostol1976} appropriately. 

We begin with two preliminary lemmas. 
\begin{Lemma}\label{prop:sn} For $n>1$, we have 
$$ \frac{\sumH(n)}{n} \ge \frac{\log n/2}{\log 3}.$$
\end{Lemma}
\begin{proof} 
We use the following refinement of Stirling's formula due to Robbins \cite[(1) and (2)]{robbins1955}:
\begin{equation}\label{robbins}
\sqrt{2\pi n} \Big( \frac{n}{e}\Big)^n e^{\frac{1}{12n+1} }
<
n!
<
\sqrt{2\pi n} \Big( \frac{n}{e}\Big)^n e^{\frac{1}{12n}} .
\end{equation}
%This appears to have been proven in \cite{EGPS1990}.
From \eqref{pillai} it is immediate that
$$S(n) = \sum_{k=2}^n H(k) \ge
\frac{1}{\log 3} \left( \log n!  -(n-1)\log 2\right) + n-1.$$
Thus, using \eqref{robbins}, we obtain, for $n>1$,
\begin{align*}
\frac{S(n)}{n} & \ge
\frac{\log n}{\log 3}- \frac{\log 2}{\log 3} +\frac{1}{n}\cdot\Big(\frac{\log (2\pi n)}{2\log 3} +
\frac{\log 2}{\log 3} - 1\Big)  \cr
&\hspace{0.75in}
+\frac{1}{(\log 3)n (12n+1)} +\Big( 1 -\frac{1}{\log 3}\Big)
\cr
&\ge 
\frac{\log n}{\log 3}- \frac{\log 2}{\log 3}
\end{align*}
 as required.
\end{proof}
We require one more lemma. 
\begin{Lemma}
Let $n$ be such that $2^{k-1}< n \le 2^k$.  Let $\beta = \log 2/2\log 3\approx 0.31546\dots $. Then
\begin{equation}\label{sn-by-n-ineq}
\beta (k-2) \le \frac{S(n)}{n} \le k.
\end{equation}
\end{Lemma}
\begin{proof}
Since $n> 2^{k-1}$, we have
$$\frac{S(n)}{n} \ge  \frac{S(2^{k-1})}{n} \ge \frac{S(2^{k-1})}{ 2^{k}}
\ge \frac{1}{2}\Big(\frac{ \log 2^{k-2}}{\log 3}\Big), $$
where we have used Lemma~\ref{prop:sn}. 
In this manner, we obtain:
$$\frac{S(n)}{n} \ge \beta(k-2),$$ 
where $\beta = \log 2/{2\log 3} $. This proves the first inequality.

Since $n$ is such that $n\leq 2^k$, then in view of Observation~\ref{obs:hnlen} in 
\S\ref{sec:formula}, we must have
$$H(m)\le k \text{ for all } m\le n.$$
%or, $$S(m)-S(m-1) \le k \text{ for all } m\le k.$$
%Thus we must have $$S(n) = \sum_{m=1}^n (S(m)-S(m-1)) \le kn,$$
%or, $$\frac{S(n)}{n}\le k.$$
The second inequality follows immediately from this. %This completes the proof of the lemma.
\end{proof}

\begin{Theorem}[A Chebyshev-type Theorem]\label{chebyshev-BB} For $n>2$,  there are constants $a$ and $A$ such that
\begin{equation*}
a\frac{n^2}{S(n)}  
\le  \pi(n)
\le
A\frac{n^2}{S(n)}  .
\end{equation*}
\end{Theorem}
\begin{proof}
 Let $n>2$, and
 $k$ be such that $2^{k-1}< n \le 2^{k}$. The inequalities \eqref{chebyshev} imply that there are constants $a^\prime$ and $A^\prime$ such that:
\begin{equation}\label{basic-ineq}
a^\prime\frac{2^{k-1}}{k-1}\le \pi(2^{k-1}) \le \pi(n) \le \pi(2^{k}) \le A^\prime \frac{2^k}{k}.
\end{equation}
Now using \eqref{sn-by-n-ineq}, we obtain
 \begin{align*}
 \pi(n) &\ge a^\prime \Big(\frac{k-2}{k-1}\Big) \frac{2^{k-1}}{k-2} 
 \ge \frac{a^\prime \beta}{4} \cdot \frac{n^2}{S(n)}\\
 \intertext{and,}
 \pi(n) &\le  {2A^\prime} \frac{n^2}{S(n)}.
 \end{align*}
This completes the proof of the theorem. 
\end{proof}
\begin{Remarks}\ 
\begin{enumerate}
\item By following the proof of \eqref{chebyshev} in Shapiro \cite[Chapter 9]{shapiro1983} or Apostol \cite[Chapter 4]{apostol1976}, we can obtain \eqref{chebyshev} in the special case when $n$ is a power of $2$; and from there, for all values of $n$.
\item  We can obtain values for $a$ and $A$ in the statement of 
Theorem~\ref{chebyshev-BB} by taking $a^\prime=1/4$ and $A^\prime=6$ (or perhaps even better values, closer to $1$). However, the purpose here is to find a suitable function that can be related to $\pi(n)$. The function is evidently $F(n)=n^2/S(n)$.
\end{enumerate}
\end{Remarks}

\begin{Theorem}\label{chebyshev-pn} For $n>2$,  there are constants $a$, $A_1$ and $A_2$,
 such that
\begin{equation*}
a {S(n)}  
\le  p_n
\le
A_1 {S(n)} + A_2 n .
\end{equation*}
\end{Theorem}
\begin{proof}
The proof is analogous to the proof of the inequalities
$$\frac{1}{6} n\log n < p_n < 12\big(n\log n + n\log(12/e)\big),$$
given by Apostol \cite[Theorem 4.7]{apostol1976}.

Let $m=p_n$, so $\pi(p_n)=n$. Let $K$ be such that $2^{K-1}<p_n =m \le 2^{K}$, and
$k$ such that $2^{k-1}< n \le 2^k$.  Clearly, $k\le K$ (since $n< p_n$). 

We begin with the first inequality.  The inequalities \eqref{basic-ineq} imply that there is a constant  $A^\prime$ such that
$$n = \pi(p_n) \le A^{\prime} \frac{2^{K}}{K} \le 2A^{\prime} \frac{m}{K} ,$$
or
$$p_n =m \ge \frac{K n}{2A^\prime}  \ge \frac{ k n}{2A^\prime}.$$
But $k\ge S(n)/n$ by
\eqref{sn-by-n-ineq}, so we obtain
$$p_n\ge a S(n),$$
where $a=1/2A^\prime$.  

For the second inequality, \eqref{basic-ineq} implies that for some $a^\prime$,
$$ p_n= m\le \frac{2n (K-1)}{a^\prime} .$$
Next, using $K-1\le \log m/\log 2$, we obtain
$$K-1\le \frac{\log m}{\log 2} \le \frac{2\sqrt{m} }{e\log 2} ,$$
so 
$$m\le \frac{4n \sqrt{m}}{a^\prime e\log 2} ,$$
which yields,
$$\sqrt{m} \le \frac{4n}{a^\prime e\log 2} .$$
Taking logs, we find that
$$\log m \le 2 \log n + 2 \log ({4}/{a^\prime e\log 2}) \le 2 k \log 2 + 2 \log ( {4}/{a^\prime e\log 2}) .$$
Finally, we put all the above together,  to find that
\begin{align*}
p_n (=m) &\le \frac{2n(K-1)}{a^\prime} \\
&\le    \frac{2n\log m }{a^\prime \log 2} \\
&\le    \frac{4 n k }{a^\prime } 
+ \frac{4  \log({4}/{a^\prime e\log 2})  }{a^\prime \log 2}  n \\
&\le \frac{4 n  }{a^\prime }  \bigg( \frac{S(n)}{\beta n} +1 \bigg)
+\frac{4  \log({4}/{a^\prime e\log 2})  }{a^\prime \log 2}  n \cr
&= A_1 S(n) + A_2 n,
\end{align*}
for some constants $A_1$ and $A_2$. This completes the proof. 
\end{proof}

To summarize, we obtained two Chebyshev-type theorems, one for $\pi(n)$ and the other for 
$p_n$.  Of course, the first such theorem came up in response to Gauss' conjecture which said the 
 constants $a$ and $A$ in \eqref{chebyshev} are both $1$. The question arises: how good are these functions
 in approximating $\pi(n)$ and $p_n$? These questions are considered in the next section.

%Pending: A Checbyshev type set of inequalities for S(n) and p_n.
%At the risk of flogging a dead horse, we give an alternate proof of one of these inequalities. 

\section{(Conjectural) formulas for prime numbers}\label{sec:conjectures}
In this section we note some conjectural formulas that are motivated by Theorems~\ref{chebyshev-BB} and \ref{chebyshev-pn} 
 and present some computational evidence. 
 %The conjectures presented here are just a few of the 
%many conjectural results that have appeared in our study. We expect to compile them elsewhere. 
%\cite{BB2019b}. 
We note here a particular constant $B$ that appears in our study:
\begin{equation}\label{ConstantB}
B:= \frac{\gamma}{\log 2} \approx 0.832746\dots ,
\end{equation}
where $\gamma$ is the Euler-Mascheroni constant.
\subsection*{A formula for $\pi(n)$}
In view of Theorem~\ref{chebyshev-BB}, the first question we 
investigate is: If
$\pi(n)$ is approximately a constant multiple of $F(n)={n^2}/{\sumH(n)}$, then what should that constant be? However, initial experiments on Sage \cite{sage} with various numerical guesses did not match the data as $n$ became large. 
However, on graphing the difference of $\pi(n)$ with $F(n)$, the error term appears to be of the same type as $F(n)$ itself. This leads to the following conjecture:
\begin{Conjecture} Let $G(n)$ be defined as
\begin{equation}\label{def:G}
G(n) =  \frac{n^2}{\sumH(n)} +  \frac{(bn)^2}{\sumH(\lfloor bn \rfloor)}.
%
 %\frac{n^2}{\sum_{k=1}^n H(k)} 
%+ \frac{(an)^2}{\sum_{k=1}^{\lfloor an \rfloor} H(k)},
\end{equation}
Then $G(n)\sim \pi(n)$, for a constant $b$, where $b$ is (approximately) $$b\approx {e^B}/{10}\approx 0.22996\dots.$$   % as $n\to\infty$. 
\end{Conjecture}
\begin{figure}[h]
\begin{center}
\includegraphics[scale=0.5]{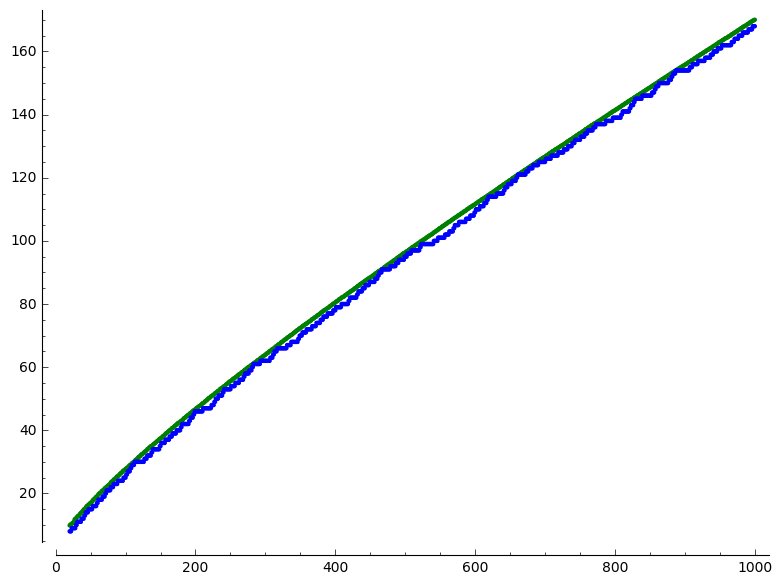}
\caption{Graph of $G(n)$ (top) against $\pi(n)$ }
\end{center}
\label{fig:pi_estimate}
\end{figure}
\subsection*{Notes}
%\begin{Remarks}
\begin{enumerate}
\item Figure~\ref{fig:pi_estimate} shows the graphs of $\pi(n)$ and $G(n)$ on the same set of axes for $n\le 1000$. The agreement is quite striking.  
\item Figure~\ref{fig:pi_estimate_li} shows two curves $\pi(n)/G(n)$ (top) and $\pi(n)/\li(n)$ (bottom) for $n=20$ to $n=90000$. It appears that $G(n)$ is a better estimate than $\li(n)$ for ``small'' values of $n$. 

\item Table~\ref{table:pi_comparison} 
compares the performance of $G(n)$ with that of $\li(n)$ for some values of $n$.  
\end{enumerate}
%\end{Remarks}

\begin{figure}[h]
\begin{center}
\includegraphics[scale=0.5]{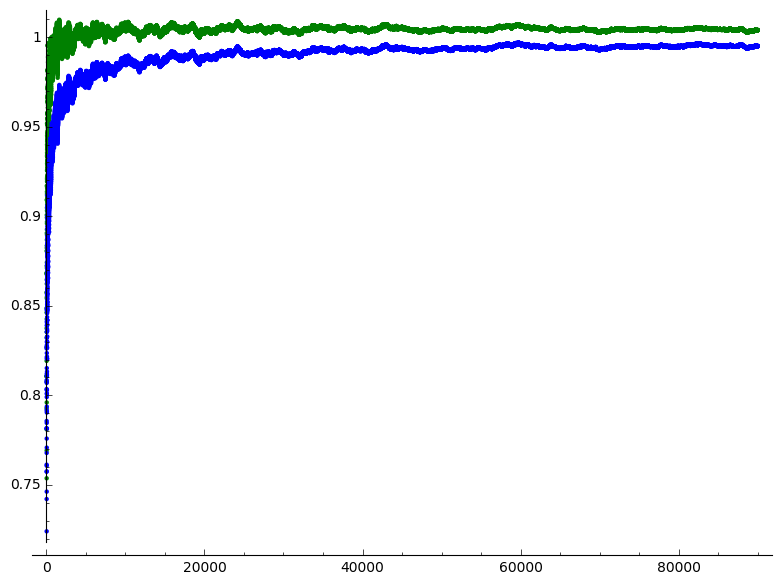}
\caption{Graphs of $\pi(n)/G(n)$  (top) and  $\pi(n)/\li(n)$ }
\end{center}
\label{fig:pi_estimate_li}
\end{figure}

\begin{table}[h]
$$
\begin{array}{| c | c | c | c | c | c  |c|}
\hline
n& \pi(n) & \lfloor G(n)\rfloor & \pi(n)- \lfloor G(n)\rfloor & \lfloor \li(n)\rfloor  & \pi(n)-\lfloor \li(n)\rfloor  & \frac{\pi(n)-G(n)}{\pi(n)}  \\
\hline
10 & 4 & 8 & -4 & 6 & -2 & -113.64 \% \cr \hline
10^2 & 25 & 27 & -2 & 30 & -5 & -10.78 \% \cr \hline
10^3 & 168 & 170 & -2 & 177 & -9 & -1.22 \% \cr \hline
10^4 & 1229 & 1222 & 7 & 1246 & -17 & 0.50 \% \cr \hline
10^5 & 9592 & 9547 & 45 & 9629 & -37 & 0.46 \% \cr \hline
10^6 & 78498 & 78340 & 158 & 78627 & -129 & 0.20 \% \cr \hline
10^7 & 664579 & 664297 & 282 & 664918 & -339 & 0.04 \% \cr \hline
\end{array}
$$
\caption{$G(n)$ against $ \li(n)$}
\label{table:pi_comparison}
\end{table}

\subsection*{A formula for the $n^\text{th}$ prime}

Given Theorem~\ref{chebyshev-pn}, one can ask how well is $p_n$ approximated by a constant times $S(n)$. It turns out that even with the constant equal to $1$, the approximation is quite good. 
Indeed, it appears that
\begin{equation}\label{pnApproxSn}
p_n \approx \sumH(n).
%{\sum_{k=1}^n H(n)}
\end{equation}
Here the values of $n$ we have computed are until $n\approx 6\times 10^7$. 
%\end{Conjecture}

\subsection*{Notes}
%\begin{Remarks}
\begin{enumerate}
\item See Figure~\ref{fig:nthprime} mentioned in the introduction for a graph of $S(n)$ and $p_n$, for $n\le 5000$. 
\item Table~\ref{table:S-nthprime1} indicates that $S(n)$ is a better approximation to $p_n$ than $n\log n$. 
\item Table~\ref{table:S-nthprime2} contains the values of $S(n)$ and $p_n$ at some large 
random values of $n$, and the relative error of the approximation. 
It appears that the relative error is increasing. 
\end{enumerate}
%\end{Remarks}

\begin{table}[h]
$$
\begin{array}{| c | c | c | c | c|c |}
\hline
n & p_n & 	S(n) & S(n)-p_n & \frac{S(n)}{n\log n} &  \frac{S(n)-p_n}{p_n}  \\
\hline
%85 & 439 & 397 & -42 & -9.57\% \\ %6719817767654 \\
%\hline
%992 & 7853 & 7569 & -284 & -3.62 \% \\ %645231121864 \\
% \hline
% 7524 & 76487 & 76008 & -479 & -0.63\% \\ %6250212454404
% \hline
%  56762 & 702853 & 713411 & 10558 & 1.50\% \\ %216332575944 \\
%\hline
%596319 & 8901749 & 9206082 & 304333 & 3.42\% \\ %5880005828068 \\
%\hline
3874958  & 65619413 & 68671533 & 3052120 &1.168215 & 4.65\% \\% 124550870335 \\
\hline
17594789 & 326260271 & 344294853 & 18034582 &1.172923  &  5.53\% \\ %766720407708
\hline
29742315 & 568063631 & 601049024 & 32985393  &1.174364  &  5.81\% \\ %663700331134
\hline 
32970915 &  633319879 & 670440504 & 37120625 & 1.174637&  5.86 \%\\ %127583088230
\hline
46262236 & 905219069 & 959827638 & 54608569  & 1.175509  & 6.03\% \\ %263573096470
\hline
54074749 & 1066983163  &1 132214258 & 65231095 &1.175901 & 6.11\% \\ %360115717215
\hline
% a& b & c & d & e  \\ 
\end{array}
$$
\caption{Comparison of $S(n)$ with $p_n$: Random values }
\label{table:S-nthprime2}
\end{table}

From the above, it seems that the approximation is slightly off. In view of Conjecture~1, we expect the following.  
\begin{Conjecture} Let $p_n$ denote the $n$th prime. Then
%\begin{equation}\label{pnApproxSn}
$$p_n \sim \frac{1}{(1+b)} \sumH(n),$$
for a constant $b$, where $b\approx {e^B}/{10}\approx 0.22996\dots$. 
%{\sum_{k=1}^n H(n)}
%\end{equation}
\end{Conjecture}
\begin{Remark}
Erdos et.al.~\cite{EGPS1990} conjectured that there is a constant $\delta$ 
such that  $S(n)\sim \delta n \log n$. These authors showed that a certain form of the Elliot--Halberstam  conjecture implies their conjecture. 
Conjecture~2 implies that $\delta \approx {1+b}\approx 1.22996\dots$. Note that 
${1}/{(1+b)}\approx 0.813\dots$.
\end{Remark}

Professor Pomerance pointed out that \eqref{pnApproxSn} is inconsistent with Conjecture~1, and 
commented that it would be interesting to perform numerical computations to conjecture the value of 
$\delta$.  From Table~\ref{table:S-nthprime2}, it appears that the number is bigger than $1.175$. But 
we were only able to compute up to $n = 6\times 10^7$. We expect the limit to be larger. 

Below, we briefly outline the steps to obtain the value of $\delta$ that follows from Conjecture~1. This motivates our statement of Conjecture~2. 

Define $S(x)$ now as $\sum_{k\le x} H(k)$.  Take $G(x)=0$ when $x<2$. 
Then \eqref{def:G} can be written as
$$G(n) =  \frac{n^2}{\sumH(n)} +  \frac{(bn)^2}{\sumH( bn )}.$$
 On inverting this, we obtain the approximation
$$\frac{n^2}{S(n)} 
\approx \sum_{k\ge 0} (-1)^k G( b^k n ).
$$
Note that for any fixed $n$ this is a finite sum.  Upon replacing $G(x)$ by $\pi(x)$, we see that from Conjecture~1, we have 
\begin{equation}\label{SnPi}
\frac{n^2}{S(n)} 
\approx \sum_{k\ge 0} (-1)^k \pi( b^k n ).
\end{equation}

Let $\widehat{S}(n)$ denote the approximation to $S(n)$ obtained from \eqref{SnPi}, that is,
$$\widehat{S}(n) := \frac{n^2}{ \sum_{k\ge 0} (-1)^k \pi( b^k n )}.$$
The computations (with $b=0.229962525551838$) in Table~\ref{table:sn-estimate} indicate that $\widehat{S}(n)$ is quite an accurate way to estimate the value of $S(n)$. This further supports Conjecture~1.  Note that we may replace $\pi(x)$ by  $\li(x)$ to estimate $S(n)$. 

\begin{table}[h]
$$
\begin{array}{| c | c | c | c | c |}
\hline
n &  	S(n) & \widehat{S}(n) & S(n)-\lfloor\widehat{S}(n)\rfloor &  \frac{S(n)-\lfloor\widehat{S}(n)\rfloor}{S(n)}  \\
\hline
%85 & 397 & 425 & -28 & -7.05289\% \\ %-0.0705289672544081 \\
%\hline
992 & 7569 & 7628  & -59 & -0.77949 \% \\ %-0.00779495309816356 \\
 \hline
7524 & 76008 & 76089 & -81 & -0.10656\% \\ %-0.00106567729712662
\hline
 56762 & 713411& 710300 & 3111 & 0.43607\% \\ %0.00436074016240288 \\
\hline
596319 & 9206082 & 9181418  & 24664 & 0.2679\% \\ %0.00267909844817806 \\
\hline
%3874958  &  68671533 & 68607810 & 63723 & 0.09279\% \\% 0.000927939092316462 \\
%\hline
17594789 & 344294853 & 344263181 & 31672 & 0.009199\% \\ %0.0000919909191904184
\hline
%29742315 &  601049024 &601256950& -207926  & -0.03459\% \\ %-0.000345938503678529
%\hline 
32970915 &  670440504 &670698724& -258220 & -0.03851 \%\\ %-0.000385149761178510
\hline
%46262236 &  959827638 &960391696&-564058  & -0.058766\% \\ %-0.000587665928411201
%\hline
54074749 & 1132214258 &1133070822& -856564 & -0.07565 \% \\ %-0.000756538785788723
\hline
% a& b & c & d & e  \\ 
\end{array}
$$
\caption{Comparison of $S(n)$ with $\widehat{S}(n)$: Random values }
\label{table:sn-estimate}
\end{table}

Multiplying both sides of \eqref{SnPi} by $\log(n)/n$, and taking (formal) limits, we obtain 
$$\lim_{n\to\infty} \frac{n\log n}{S(n)} \approx \frac{1}{1+b}\approx 0.813\dots,$$ 
which suggests that $\delta \approx {1+b}\approx 1.22996\dots$. 

From the above we see that Conjecture~2 is consistent with Conjecture~1.

%\begin{table}[h]
%$$
%\begin{array}{| c | c | c | c | c |}
%\hline
%n & p_n & 	S(n) & S(n)-p_n &  \frac{S(n)-p_n}{p_n}  \\
%\hline
%%85 & 439 & 397 & -42 & -9.57\% \\ %6719817767654 \\
%%\hline
%%992 & 7853 & 7569 & -284 & -3.62 \% \\ %645231121864 \\
%% \hline
%% 7524 & 76487 & 76008 & -479 & -0.63\% \\ %6250212454404
%% \hline
%%  56762 & 702853 & 713411 & 10558 & 1.50\% \\ %216332575944 \\
%%\hline
%%596319 & 8901749 & 9206082 & 304333 & 3.42\% \\ %5880005828068 \\
%%\hline
%3874958  & 65619413 & 68671533 & 3052120 & 4.65\% \\% 124550870335 \\
%\hline
%17594789 & 326260271 & 344294853 & 18034582 & 5.53\% \\ %766720407708
%\hline
%29742315 & 568063631 & 601049024 & 32985393  & 5.81\% \\ %663700331134
%\hline 
%32970915 &  633319879 & 670440504 & 37120625 & 5.86 \%\\ %127583088230
%\hline
%46262236 & 905219069 & 959827638 & 54608569  & 6.03\% \\ %263573096470
%\hline
%54074749 & 1066983163  &1 132214258 & 65231095 & 6.11\% \\ %360115717215
%\hline
%% a& b & c & d & e  \\ 
%\end{array}
%$$
%\caption{Comparison of $S(n)$ with $p_n$: Random values }
%\label{table:S-nthprime2}
%\end{table}
%

%
%What is remarkable is that the two curves seem to mirror each other. Here we take a smaller range to see the mirroring in more detail. 
%\begin{figure}[h]
%\includegraphics[scale=0.75]{mirroring_li_sumh.png}
%\caption{Graphs of $\pi(n)/\pi_H(n)$  (top) and  $\pi(n)/\li(n)$ }
%\label{fig:mirroring_li_sumh}
%\end{figure}
%
\subsection*{Prime gaps}
We end this section with a few remarks about the prime gap. Let $g_n$ denote the prime gap, that 
is, $g_n=p_{n+1}-p_n.$ Given that $S(n)$ approximates $p_n$, it is natural to ask whether  $g_n$ is 
approximated by $H(n+1)$. On the other hand, the prime gap is notorious for its irregularity, and one cannot expect much in this regard.
Nevertheless, it seems that on average, $H(n+1)$ does quite a good job of approximating $g_n$. Indeed, Table~\ref{table:gaps} gives a few 
values of the following function:
\begin{equation}\label{prime_gap_ave}
S_{\Delta}(k)=\frac{1}{2^k} \sum_{2^k<m \le 2^{k+1} } \frac{g_m}{H(m+1)}.
\end{equation}
\begin{table}[h]
$$
\begin{array}{| c | c || c| c| }
\hline
k& S_\Delta(k) & k& S_\Delta(k) \\
\hline
1 & 1.16666666666667 & 10 & 1.02073447927940 \\
\hline
2 & 1.08333333333333 & 11 & 1.01134933317550 \\
\hline
3 & 1.23333333333333  & 12 & 1.00132388905581 \\
\hline
4 & 1.11041666666667  & 13 & 0.994501291351865 \\
\hline
5 & 1.08541666666667 & 14 & 0.988054230925281 \\
\hline
6 & 1.08377976190476  & 15 & 0.982638771673053 \\
\hline
7 & 1.03947792658730 & 16 & 0.976617044103504 \\
\hline
8 & 1.05543154761905 & 17 & 0.971171169038482 \\
\hline
9 & 1.04229290674603 & 18 & 0.966205359427567  \\
\hline
\end{array}
$$
\caption{$S_\Delta(k)$ for $k=1,\dots,18$}
\label{table:gaps}
\end{table}
In view of Conjecture~2, we expect the limit to be $1/(1+b)\approx 0.813\dots$. 
%We leave it to you, dear reader, to draw your own conclusions from Table~\ref{table:gaps}.

\subsection*{Acknowledgements}
We thank Carl Pomerance (Dartmouth) and Manjil P. Saikia (Vienna) for useful comments.  All numerical computations and graphics have been generated using Sage \cite{sage}.
The research of the second author was partially supported by the
Austrian Science Fund (FWF),  grant~F50-N15, in the framework of the
Special Research Program ``Algorithmic and Enumerative Combinatorics".

%\bibliography{ref}{}
%\bibliographystyle{abbrv}

%\end{document}

\end{document}